\documentclass{amsart}

\usepackage{amssymb, amsmath, amsthm, hyperref, euscript}

\usepackage{amscd}

\usepackage[english]{babel}

\author{Rafael Torres}

\title[Involutions, $\mathcal{F}$-structures, entropy, and positive scalar curvature]{On entropies, $\mathcal{F}$-structures, and scalar curvature of certain involutions}

\address{Scuola Internazionale Superiori di Studi Avanzati\\ Via Bonomea 265\\34136\\Trieste\\Italy}

\email{rtorres@sissa.it}

\date{January 10th, 2014}


\theoremstyle{plain}

\newtheorem{theorem}[equation]{Theorem}

\newtheorem{corollary}[equation]{Corollary}

\newtheorem{proposition}[equation]{Proposition}

\newtheorem{remark}[equation]{Remark}

\newtheorem{lemma}[equation]{Lemma}

\theoremstyle{definition}

\newtheorem{definition}[equation]{Definition}

\newcommand{\Q}{\mathbb{Q}}

\newcommand{\R}{\mathbb{R}}

\newcommand{\Z}{\mathbb{Z}}

\newcommand{\N}{\mathbb{N}}

\newcommand{\C}{\mathbb{C}}

\begin{document}

\maketitle

\emph{Abstract}: In this short note, we analyze geometric properties of orbit spaces of certain involutions in dimensions four, five, and six. We consider constructions of $\mathcal{F}$-structures on manifolds of dimension at least four that allows us to study minimal entropy, minimal volume, collapse with bounded curvature, and sign of the Yamabe invariant, and its vanishing, building on work of Paternain-Petean. The existence of Riemannian metrics of vanishing topological entropy on the orbit spaces is investigated as well.

\section{Introduction and main results}

For a closed connected Riemannian manifold $(M, g)$ with finite fundamental group, we consider the following invariants 
\begin{itemize}
\item existence of an $\mathcal{F}$-structure \cite{[G], [CG], [PP1]}, 
\item minimal entropy $h(M)$ \cite{[PP1]},
\item topological entropy $h_{top}(g)$ \cite{[Ma], [GPa]}, 
\item minimal volumes $MinVol(M)$ and $Vol_K(M)$ \cite{[G]}, and collapsing with sectional curvature $K$ bounded form below \cite{[CG], [PP1]}, 
\item sign of the scalar curvature, i.e., sign of the Yamabe invariant $\mathcal{Y}(M)$ \cite{[Sc]}, and
\item vanishing $\mathcal{Y}(M) = 0$ and its (non-)realization \cite{[L2]}.
\end{itemize}

Due to rescalling properties of these invariants, in order to obtain interesting data we assume the normalization $Vol((M, g)) = 1$ of the volume throughout the manuscript. A geometric expression of the topological entropy of $(M, g)$ is given by a formula due to Ma\~n\'e \cite{[Ma]} in terms of geodesic arcs as the following mean value. Consider any two points $p, q\in M$, a positive number $T > 0$, and let $n_T(p, q)$ be the number of geodesic arcs joining $p$ and $q$ with length less or equal to $T$. Ma\~n\'e's formula is 
\begin{equation}
h_{top}(g) = \underset{T\rightarrow \infty}\lim \frac{1}{T} \log\int_{M\times M} n_T(p, q)\, \mathrm{d}p \mathrm{d}q .
\end{equation}

The minimal entropy $h(M)$ of $(M, g)$ is the infimum of the topological entropy of the geodesic flow of a smooth metric $g$ on $M$. The simplicial and spherical volumes, as well as the volume entropy vanish for manifolds with finite fundamental group.\\

Paternain and Petean showed that the presence of an $\mathcal{F}$-structure on a manifold has interesting geometric implications. An $\mathcal{F}$-structure is a generalization of a torus action on a manifold introduced by Cheeger-Gromov \cite{[CG]} (see Section \ref{Section S} for a precise definition). It is proven in \cite[Theorem A]{[PP1]} that in the presence of such a structure the minimal entropy vanishes.\\

The existence of an $\mathcal{F}$-structure also implies the vanishing of the minimal volume \begin{equation} Vol_K(M):= \underset{g}{\inf} \{Vol(M, g) : K_g \geq -1 \},\end{equation} \cite[Theorem 7.2]{[PP1]} where $K_g$ is the sectional curvature of a metric $g$, and the manifold collapses with sectional curvature bounded from below, as previously studied by Cheeger-Gromov \cite{[CG]}. Recall that a smooth manifold $M$ \emph{collapses with (Ricci/scalar) curvature bounded from below} if and only if there is a sequence of smooth Riemannian metrics $\{g_j\}$ for which the (Ricci/scalar respectively) sectional curvature is uniformly bounded from below, but their volumes $\{Vol(M, g_j)\}$ approach zero as $j \rightarrow \infty$. \\

There is the following relation between collapsing and the Yamabe invariant, whose definition we first recall. Let $\gamma:= [g] = \{ug : M\overset{u}{\rightarrow} \R^+\}$ be a conformal class of Riemannian metrics on $(M, g)$. The Yamabe constant of $(M, \gamma)$ is \begin{equation} \mathcal{Y}(M, \gamma):= \underset{g\in \gamma}{\inf} \frac{\int_M scal_g dvol_g}{(Vol(M, g))^{2/n}}.\end{equation} The Yamabe invariant of $M$ is defined as \begin{equation} \mathcal{Y}(M):= \underset{\gamma}{\sup} \mathcal{Y}(M, \gamma). \end{equation} If the dimension of the manifold is at least three and it admits an $\mathcal{F}$-structure, then the manifold collapses with bounded scalar curvature, i.e., $\mathcal{Y}(M) \geq 0$ \cite[Theorem 7.2]{[PP1]}. Petean has shown that the Yamabe invariant of a simply connected closed manifold of dimension greater or equal to five is nonnegative \cite{[Pe]}. LeBrun \cite[Theorem A]{[L2]} determined the sign of the Yamabe invariant of the underlying smooth 4-manifold of a complex algebraic surface in terms of its Kodaira dimension, and studied collapse of elliptic surfaces $E(n)$ \cite[Section 5]{[L2]}. Our first result extends this work to inequivalent smooth structures on $E(n)$ that are built using generalized fiber sums and logarithmic transformations as in Szab\'o \cite{[Sz]}, and by using Fintushel-Stern's Knot surgery \cite{[FS]}.

\begin{theorem}{\label{Theorem Extra}} Every element of the infinite set $\{E(n)_K:$ K is a knot in $S^3\}$ of pairwise nondiffeomorphic 4-manifolds constructed using Fintushel-Stern's Knot surgery that have the homeomorphism type of the elliptic surface $E(n)$ with $n\geq 2$ has an $\mathcal{F}$-structure. Consequently, for every such 4-manifold $M$ we have \begin{equation} h(M) = 0 = Vol_K(M), \end{equation} and $M$ collapses with curvature bounded from below. 

Moreover, the Yamabe invariant satisfies \begin{equation} \mathcal{Y}(M) = 0 \end{equation} and it is realized by $M$, i.e., there exists a scalar-flat Riemannian metric $(M, g)$ if and only if $M$ is diffeomorphic to the K3 surface $E(2)$.

\end{theorem}

Conditions for the Yamabe invariant to vanish in the presence of an $\mathcal{F}$-structure have been previously studied by Su\'arez-Serrato \cite{[P0]}. Given our motivation regarding the existence of $\mathcal{F}$-structures, we study constructions on a wide range of manifolds in Section \ref{Section S}.\\ 

Paternain \cite{[GPa1]} and Paternain-Petean \cite{[PP1]} have shown as well that the topological type of $(M, g)$ is restricted under the existence of a Riemannian metric of zero topological entropy. They have shown that in dimension four the only possible homeomorphism types are \begin{equation} \{S^4, \mathbb{CP}^2, S^2\times S^2, \mathbb{CP}^2\#\overline{\mathbb{CP}^2}, \mathbb{CP}^2\#\mathbb{CP}^2\}\end{equation} and in dimension five, the only possible diffeomorphism types are \begin{equation}\{S^5, S^3\times S^2, S^3\widetilde{\times} S^2, SU(3)/SO(3)\}. \end{equation}

The nontrivial $S^3$-bundle over $S^2$ is being denoted by $S^3\widetilde{\times} S^2$. The complex projective space is denoted $\mathbb{CP}^2$, and $\overline{\mathbb{CP}^2}$ is the underlying smooth manifold with the reverse orientation. We proceed to extend these results to orbit spaces of involutions on these manifolds by $\Z/2$, as follows.

\begin{theorem}{\label{Theorem 4D}} Every closed orientable smooth 4-manifold with fundamental group of order two is homeomorphic to a smooth manifold $M$, which admits an $\mathcal{F}$-structure. Consequently, \begin{equation} h(M) =  0 = Vol_K(M),\end{equation} and $M$ collapses with sectional curvature bounded from below. In particular, \begin{equation}\mathcal{Y}(M) \geq 0.\end{equation}

Suppose $M$ has finite nontrivial fundamental group. There exists a Riemannian metric $(M, g)$ such that \begin{equation}h_{top}(g) = 0\end{equation} if and only if $M$ is homeomorphic to a rational homology 4-sphere with fundamental group of order two.

In particular, both total spaces of the two orientable $S^2$-bundles over $\mathbb{RP}^2$ admit such a Riemannian metric of zero topological entropy.
\end{theorem}

Paternain-Petean \cite{[PP3]} and Su\'arez-Serrato \cite{[P]} have studied the existence of the properties in the conclusion of Theorem \ref{Theorem 4D} for 4-manifolds that admit a model geometry in the sense of Thurston using the comprehensive descriptions of these geometric manifolds given by Hillman \cite{[H]}.\\

Regarding dimension five, we have the following. An orientable 5-manifold $M$ is of fibered type if and only if $\pi_2(M)$ is a trivial $\Z[\pi_1(M)]$-module. The orbit spaces of orientation-preserving $\Z/2$-involutions on the connected sums \begin{equation} S^5\#k_1(S^3\times S^2)\# k_2(S^3\widetilde{\times} S^2)\end{equation} for $k_1\geq 0, k_2\in \{0, 1\}$ that act trivially on the second homology group were classified by Hambleton-Su in \cite[Theorem 3.1]{[HS]}. The orbit spaces include four smooth manifolds that are homotopy equivalent to the real projective 5-space $\mathbb{RP}^5$, and that realize two homeomorphism types as involutions on $S^5$ \cite{[LMe], [HS]}.

\begin{theorem}{\label{Theorem 3}} Every closed smooth fibered 5-manifold with fundamental group of order two and torsion-free second homology $M$ admits an $\mathcal{F}$-structure, and consequently  \begin{equation} h(M) =  0 = Vol_K(M),\end{equation} and $M$ collapses with sectional curvature bounded from below. 

Moreover, there is a Riemannian metric of positive scalar curvature on $M$, i.e., the Yamabe invariant satisfies \begin{equation}\mathcal{Y}(M) > 0.\end{equation}
Suppose there exists a Riemannian metric $g$ on $M$ with \begin{equation}h_{top}(g) = 0.\end{equation} Then $b_2(M) \leq 1$. Up to homotopy equivalence, a metric of zero topological entropy exists if and only if $M$ is homotopy equivalent to $\mathbb{RP}^5$, an $S^2$-bundle over $\mathbb{RP}^3$ or the orbit space of $S^3\widetilde{\times} S^2$ by a $\Z/2$-involution.

 If furthermore $\omega_2(M) = 0$, then $M$ admits a Riemannian metric of zero topological entropy metric if and only if it is diffeomorphic to $S^2\times \mathbb{RP}^3$.
\end{theorem}

The cut-and-paste constructions of Gromov-Lawson \cite[Theorem A]{[GL]} and Schoen-Yau \cite[Corollary 4]{[SY]} imply the claim in Theorem \ref{Theorem 3} regarding the scalar curvature, once the structure results of Hambleton-Su are invoked \cite[Theorem 3.7]{[HS]}. While the classification of simply connected manifolds of dimension greater or equal to five that admit a metric of positive scalar curvature has been obtained by Stolz \cite{[St]}, the existence question for manifolds with finite fundamental group as been studied extensively, and it is not yet settled (see Botvinnik-Rosenberg \cite{[BR2]}). Theorem \ref{Theorem 3} is a contribution in this direction, and in particular it answers in the affirmative the propagation question of Kwasik-Schultz \cite{[KS]} for fibered 5-manifolds with fundamental group of order two.\\

In dimension six, using the classification of Smale \cite{[Sm]} of 2-connected 6-manifolds and the classification of Hambleton \cite{[IH]} of free differentiable orientation-preserving involutions on a connected sum $k(S^3\times S^3)$, we obtain the following result.

\begin{theorem}{\label{Theorem 6D}} Every closed orientable smooth 6-manifold with trivial second homotopy group and fundamental group of order at most two is homotopy equivalent to a smooth manifold $M$, which admits an $\mathcal{F}$-structure. Consequently, \begin{equation}  h(M) =  0 = Vol_K(M),\end{equation} and $M$ collapses with sectional curvature bounded from below.

Moreover, $M$ admits a Riemannian metric of positive scalar curvature, i.e., \begin{equation}\mathcal{Y}(M) >  0.\end{equation}

Finally, there exists a Riemannian metric $(M, g)$ such that \begin{equation}h_{top}(g) = 0\end{equation} if and only if $M$ is homotopy equivalent to $S^6, S^3\times S^3$ or to an $S^3$-bundle over $\mathbb{RP}^3$. For the latter two cases, $Min Vol(M) = 0$ and $M$ collapses with bounded sectional curvature.
\end{theorem}

The minimal volume is defined \begin{equation} MinVol(M):= \underset{g}{\inf}\{Vol(M, g) | |K_g|\leq 1\}.\end{equation}  Cheeger-Gromov \cite{[CG]} have proven that in the presence of a polarized $\mathcal{F}$-structure, this minimal volume vanishes and $M$ collapses with bounded sectional curvature. Since $Min Vol(M) = 0$ implies that the Euler characteristic of $M$ is zero, the last implication of the theorem is optimal. Bessi\`eres \cite{[B]} and Kotschick \cite{[K]} have shown the dependance of the minimal volume on the smooth structure. Kotschick's examples include pairs of homeomorphic yet non-diffeomorphic manifolds for which a smooth structure collapses with bounded sectional curvature, while the other does not.\\

We now switch gears, and consider involutions with inequivalent smooth structures as in the following definition.

\begin{definition}{\label{Definition Inv}} Let $A_G$ and $B_G$ be closed smooth n-manifolds with fundamental group $\pi_1(A_G) = G = \pi_1(B_G)$, and consider the involutions
\begin{equation}C\overset{\varphi_{A_G}}\longrightarrow A_G\end{equation} and \begin{equation}C\overset{\varphi_{B_G}}\longrightarrow B_G,\end{equation} where $C$ is the universal cover. The involution \emph{$\varphi_{A_G}$ is an exotic copy of $\varphi_{B_G}$} if there exists a $G$-equivariant homeomorphism $A_G\rightarrow B_G$, but no such diffeomorphism; in particular, the orbit space $B_G$ has inequivalent smooth structures. If such an exotic involution exists, we say that $C$ admits an \emph{exotic $G$-involution}. If the orbit space $B_G$ admits infinitely many inequivalent smooth structures, we say that $C$ admits an infinite set of exotic free $G$-involutions.

\end{definition}

Building upon recent progress on unveiling exotic smooth structures on 4-manifolds \cite{[FPS], [ABBKP], [BK], [AP1]}, we obtained the following result.

\begin{theorem}{\label{Theorem 1}} Let $G$ be a nontrivial finite group that acts freely on $S^3$. There is an infinite set of exotic free $G$-involutions on the connected sums \begin{equation} (da + 1)\mathbb{CP}^2 \# (db + 1)\overline{\mathbb{CP}^2}\# (d - 1)(S^2\times S^2),\end{equation}
where $d = |G|$, $a$ is an odd integer with , and $a, b\in \Z$ satisfy the numerical relations 2a = 0 mod 4, and\begin{equation} 5a - b + 4 \geq 0, a\leq b - 1.\end{equation}

There is an infinite set of exotic free $G$-involutions on the connected sums \begin{equation} dqE(2)\#(dq + 2dk - 1)(S^2\times S^2),\end{equation}
where $q\in \N$ and $k\in \N\cup \{0\}$. 
\end{theorem}

The contribution of Theorem \ref{Theorem 1} is a considerable increase on the number of 4-manifolds known to admit such involutions, since involutions of this nature were known to exist. Examples of involutions of this kind were presented by Ue \cite{[Ue]}, and Hanke-Kotschick-Wehrheim \cite{[HKW]}; LeBrun produced examples for $G = \Z/2$  \cite{[L3]}. Infinite families of orientation-preserving exotic $\Z/d$-involutions were given by Suvaina \cite{[S]} and Fintushel-Stern-Sunukjian \cite[Theorem 1]{[FSN]}. Sung constructed exotic $\Z/2\oplus G$-involutions in \cite{[CS]}. Unlike the group actions involved in Theorem \ref{Theorem 1}, the actions in \cite{[CS]} need not be free. A key ingredient in the proof of Theorem \ref{Theorem 1} is the recent work of Baykur-Sunukjan, which shows that the inequivalent smooth structures on simply connected manifolds under consideration become diffeomorphic after taking a connected sum with $S^2\times S^2$ \cite{[BS]}.\\

The aforementioned results of Gromov-Lawson \cite{[GL]}, and Schoen-Yau \cite{[SY]}, along with results of Witten, Taubes, and Kotschick-Morgan-Taubes on Seiberg-Witten theory \cite{[W], [Ta], [KMT]} allow us to conclude the following geometric properties of these involutions.

\begin{theorem}{\label{Theorem 2}} The orbit spaces of the infinite set of exotic $G$-involutions on \begin{equation} (da + 1)\mathbb{CP}^2 \# (db + 1)\overline{\mathbb{CP}^2}\# (d - 1)(S^2\times S^2),\end{equation} of Theorem \ref{Theorem 1} do not admit a metric of positive scalar curvature, while their universal covers do. In particular, the Yamabe invariant changes sign as one passes from the orbit space to the universal cover.

There are infinitely many inequivalent smooth structures on the homeomorphism type of \begin{equation} (da + 1)\mathbb{CP}^2 \# (db + 1)\overline{\mathbb{CP}^2}\# (d - 1)(S^2\times S^2),\end{equation} and on their orbit $G$-spaces such that neither admit a metric of positive scalar curvature for $G\in \{\Z/p, \Z/p\oplus \Z/q\}$.

\end{theorem}

The standard smooth structures on the manifolds of Theorem \ref{Theorem 2} and on their orbit $G$-spaces as connected sums have Riemannian metrics of positive scalar curvature, and their Yamabe invariants are positive. The phenomena described in Theorem \ref{Theorem 2} was previously known to exist, and the result provides infinite sets of simply connected and nonsimply connected examples that answer in the negative the propagation question of Kwasik-Schultz \cite{[KS]} that was mentioned before. LeBrun displayed a smooth structure on a 4-manifold with fundamental group $\Z/2$ that has negative Yamabe invariant, while its universal cover has positive Yamabe invariant \cite[Theorem 1]{[L3]}.  In \cite[Theorem 9]{[HKW]}, Hanke-Kotschick-Wehrheim showed the existence of 4-manifolds that have $G$ as in the hypothesis of Theorem \ref{Theorem 1} as fundamental group, which do not admit a Riemannian metric of positive scalar curvature although their universal covers do admit such a metric. These were the first counterexamples to a conjecture of Rosenberg, which states that a closed manifold with finite fundamental group of order admits a metric of positive curvature if and only if so does its universal cover \cite{[Ro]}. Theorem \ref{Theorem 2} extends these results in the sense that it includes infinitely many different diffeomorphism types of orbit spaces, and increases the number of homeomorphism types of the universal covers. Moreover, LeBrun has shown that the Yamabe invariant of a surface of general type is negative \cite{[L2]}. The examples of negative Yamabe invariant of Theorem \ref{Theorem 2} are symplectic, yet they need not be K\"ahler.

\subsection{Acknowledgements} We thank Gabriel Paternain, Jimmy Petean, Yang Su, and Pablo Su\'arez-Serrato for useful conversations and e-mail correspondence. The Simons Foundation is gratefully acknowledged for its support. We thank CIMAT for its hospitality during the production of part of this manuscript.

\section{$\mathcal{F}$-structures, circle sums, and topological entropy of submersions}

\subsection{Constructions of $\mathcal{F}$-structures}{\label{Section S}} In this section, we study constructions of $\mathcal{F}$-structures. In particular, address a question of Paternain-Petean \cite[Question 2]{[PP2]}, and extend work of Botvinnik-Rosenberg \cite{[BR1]}. We recall the notion of an $\mathcal{F}$-structure, which was introduced by Cheeger-Gromov \cite{[G], [CG]} expressed in terms of sheaves. We work with the equivalent definition given in \cite[Section 2]{[PP2]} cf. \cite[Definition 5.3]{[PP1]}.

\begin{definition}{\label{Definition F}}An \emph{$\mathcal{F}$-structure} on a smooth closed manifold $M$ is given by 
\begin{itemize}
\item a finite open cover $\{U_1, \ldots, U_N\}$ of $M$;
\item a finite Galois covering $\pi_i: \widetilde{U_i}\rightarrow U_i$ with $\Gamma_i$ a group of deck transformations for $1\leq i \leq N$;
\item a smooth effective torus action with finite kernel of a $k_i$-dimensional torus 
\begin{equation}
\phi_i: T^{k_i}\rightarrow Diff(\widetilde{U_i})
\end{equation}
for $1\leq i \leq N$;
\item a representation $\Phi_i: \Gamma_i \rightarrow Aut(T^{k_i})$ such that
\begin{equation}
\gamma(\phi_i(t)(x)) = \phi_i(\Gamma_i(\gamma)(t))(\gamma x)
\end{equation}
for all $\gamma \in \Gamma_i$, $t\in T^{k_i}$, and $x\in \widetilde{U}_i$;
\item for any subcollection $\{U_{i_1}, \ldots, U_{i_l}\}$ that satisfies $U_{i_1\cdots i_l}:= U_{i_1} \cap \cdots \cap U_{i_l} \neq \emptyset$, the following compatibility condition holds: let $\widetilde{U}_{i_1\cdots i_l}$ be the set of all points $(x_{i_1}, \ldots, x_{i_l}) \in \widetilde{U}_{i_1}\times \cdots \times \widetilde{U}_{i_l}$ such that $\pi_{i_1}(x_{i_1}) = \cdots = \pi_{i_l}(x_{i_l})$. The set $\widetilde{U}_{i_1\cdots i_l}$ covers $\pi^{-1}(U_{i_1\cdots i_l}) \subset \widetilde{U}_{i_1\cdots i_l}$ for all $1\leq j \leq l$. It is required that $\phi_{i_j}$ leaves $\pi^{-1}_{i_j}(U_{i_1\cdots i_l})$ invariant, and it lifts to an action on $\widetilde{U}_{i_1\cdots i_l}$ such that all lifted actions commute.
\item An $\mathcal{F}$-structure is called a \emph{$\mathcal{T}$-structure} if the Galois coverings $\pi_i: \widetilde{U}_i \rightarrow U_i$ in Definition \ref{Definition F} can be taken to be trivial for every $i$.

\end{itemize}
\end{definition}

Paternain-Petean \cite[Theorem 5.10]{[PP1]} showed that every elliptic surface admits a $\mathcal{T}$-structure. We first observe the existence of these structures on inequivalent smooth structures on the homeomorphism type of elliptic surfaces $E(n)$ built using cut-and-paste constructions of fiber sum \cite{[Go1]}, torus surgeries \cite[Section ]{[FPS]}, and Fintushel-Stern's Knot surgery \cite{[FS]}. The following proposition extends the result of Paternain-Petean to noncomplex and nonsymplectic examples.

\begin{proposition}{\label{Proposition Knot}} For any knot $K \subset S^3$ the homotopy elliptic surface $E(n)_K$ obtained from Knot surgery admits a $\mathcal{T}$-structure.

Let $E(n)_s$ be a homotopy elliptic surface obtained by applying torus surgeries to a fiber sum of $E(n)$ and the 4-torus. Then $E(n)_s$ admits a $\mathcal{T}$-structure.
\end{proposition}

\begin{proof} Consider a knot $K\subset S^3$ and its tubular neighborhood $\nu(K)$. Let $F\subset E(n)$ be a regular fiber, and $\nu(F)$ a tubular neighborhood that is diffeomorphic to $T^2\times D^2$. The homotopy surface $E(n)_K$ obtained from Fintushel-Stern's Knot surgery is \begin{equation} E(n)_K:= (E(n) - \nu(F)) \cup (S^3 - \nu(K) \times S^1)\end{equation} where the longitude of $K$ is identified with the meridian $\{x\}\times \partial D^2 \subset T^2\times D^2 = \nu(F)$ of the tubular neighborhood of the fiber.

To show that the manifold $E(n)_K$ admits a $\mathcal{T}$-structure, we equip each building block $E(n) - \nu(F)$ and $S^3 - \nu(K)\times S^1$ with $\mathcal{T}$-structures that commute with each other. This assures that the two structure glue together to provide a $\mathcal{T}$-structure on $E(n)_K$. For the piece $E(n) - \nu(F)$, we choose the $\mathcal{T}$-structure built by Paternain-Petean for the proof of \cite[Theorem 5.10]{[PP1]}. On small open subsets covering $\nu(F)\subset E(n)$, the piece is the total space of a fiber bundle with structure group $\{Id, x\mapsto -x\}$ for $x\in F = T^2$, and Paternain-Petean's $\mathcal{T}$-structure is actually polarized. On the block $S^3 - \nu(K)\times S^1$ take the circle action provided by the circle factor $\{pt\}\times S^1$. The torus and circle action commute with each other, and we obtain a $\mathcal{T}$-structure on $E(n)_K$ (see \cite[Section 3]{[PSS]}).

The construction of the infinite set $\{E(n): s\in \N\}$ is done as follows. Consider the 4-torus $S^1_x \times S^1_y\times S^1_a\times S^1_b$, where our notation indicate the loops generating the fundamental group $\pi_1(T^4) = \Z x\oplus \Z y \oplus \Z a \oplus \Z b$.  Consider the embedded 2-tori of self-intersection zero $T:= S^1_x\times S^1_y\times \{x_a\}\times \{x_b\}$, $T_a:= \{x_x\} \times S^1_y\times S^1_a\times \{x_b\}$, and $T_b:= \{x_x\} \times S^1_y\times \{x_a\}\times S^1_b$ inside the 4-torus. First, construct the fiber sum of $E(n)$ and $T^4$ along $F$ and $T$ \begin{equation} Z: = (E(n) - \nu(F)) \cup (T^4 - \nu(T)).\end{equation} Notice that, since $\pi_1(E(n) - \nu(F)) = \{1\}$, the Seifert-van Kampen theorem implies $\pi_1(Z) = \Z a \oplus \Z b$. The tori $T_a$ and $T_b$ are embedded in $Z$, where they also have trivial self-intersection, and the homotopy classes of the loops $S^1_a\subset T_a$ and $S^1_b\subset T_b$ generate the fundamental group. 

We proceed to apply a multiplicity 1 logarithmic transformation on $T_a$ along $S^1_a$ to produce a manifold $Z_a$, followed by a multiplicity s logarithmic transformation $T_b$ along $S^1_b$ with $s\in \N$ to produce a manifold $Z_{ab}$. We carve out a tubular neighborhood of $T_a$ diffeomorphic to $S^1_y\times S^1_a\times D^2$, and glue back an $S^1\times S^1\times D^2$ by an orientation reversing diffeomorphism of the 3-torus boundary $\varphi: S^1_y\times S^1_a\times \partial D^2\rightarrow S^1\times S^1\times \partial D^2$, \begin{equation} Z_a:= (Z - \nu(T_a)) \cup_{\varphi} (T^2\times D^2) \end{equation} where the isotopy class of $\varphi$ is determined by an element of $SL(3, \Z)$ (see \cite{[GS], [Sz], [FPS]} for details). The torus $T_b$ is disjoint from the previous surgery, and it is therefore icontained in $Z_a$. We can proceed to apply a logarithmic transformation to it along the curve $S^1_b$ to produce a manifold $Z_{abs}$, as it was mentioned. Denote the meridians of $T_a$ and $T_b$ by $\mu_{T_a}$ and $\mu_{T_b}$ respectively. The logarithmic transformation of multiplicity one identifies the curve $S^1_a + \mu_{T_a}$ with the meridian of $T^2\times D^2$, while the logarithmic transformation of multiplicity $s$ identifies the curve $S^1_b + s\cdot \mu_{T_b}$ with the meridian of $T^2\times D^2$.

We claim that the manifolds $Z_{abs}$ admit a $\mathcal{T}$-structure for all $s\in \N$, and proceed to prove the claim by endowing each building block in the construction of $Z_{abs}$ with $\mathcal{T}$-structures that commute along their common boundaries as follows. On the piece $E(n) - \nu(F)$ take the $\mathcal{T}$-structure of Paternain-Petean that we have used before, and name it $\sigma_{PP}$. Take the polarized $\mathcal{F}$-structure on the 4-torus given by the circle action on the first factor $S^1_x\times\{x_y\} \times \{x_a\} \times \{x_b\}$, and call it $\sigma_1$. Notice that $\sigma_{PP}$ commute with $\sigma_1$ along the 3-torus boundary. Now, take $\sigma_i$ for $i\in \{2, 3\}$ to be $\mathcal{T}$-structure on the block $T^2\times D^2$ given the circle action on $\{pt\}\times D^2$ that fixes $(0, 0)$. These actions commute with $\sigma_1$. Thus, $\sigma_{PP}$, and $\sigma_i$ paste together to give a global $\mathcal{T}$-structure on $Z_{abs}$ for each value of $s\in \N$ used as multiplicity of the last logarithmic transformation.

We now argue that the manifolds $Z_{abs}$ are homeomorphic to $E(n)$. We first fix $s\in \N$ and $n$. The characteristic numbers of $Z$ are those of $E(n)$, that is, $c_2(Z) = 12n$ and $\sigma(Z) = - 8n$, and $\omega_2(Z) = 0$. Since the logarithmic transformations performed preserve the Euler characteristic, the signature, and the second Stiefel-Whitney class, Freedman's theorem \cite{[F]} implies that $Z_{abs}$ is homeomorphic to $E(n)$. We rename $Z_{abs}:= E(n)_s$.

Finally, we point out that the infinite set $\{E(n)_s: s\in \N\}$ consists of pairwise nondiffeomorphic elements that are distinguished by their Seiberg-Witten invariants (see \cite{[Sz]}).

\end{proof}

The following proposition constructs certain manifolds with cyclic fundamental group that are equipped with a $\mathcal{T}$-structure.

\begin{proposition}{\label{Proposition C1}} Suppose $n\geq 3$. The manifold  \begin{equation} M_{\varphi}:= (D^{n - 1}\times S^1) \cup_{\varphi} (S^{n - 2}\times D^2)\end{equation} that is obtained by identifying the common boundaries using the diffeomorphism \begin{equation}\varphi: S^{n - 2}\times S^1\rightarrow S^{n - 2}\times S^1\end{equation}\begin{equation} (x, e^{i\theta})\mapsto (x, e^{ip\theta})\end{equation} for $x\in S^{n - 2}$, and $e^{i\theta} \in S^1$ admits a $\mathcal{T}$-structure.
\end{proposition}

\begin{proof} Consider the circle actions $\sigma_{n - 1}$ and $\sigma_2$ on the disks $D^{n - 1}$ and $D^2$ by rotation fixing their respective origins. The induced circle action on $\partial D^2 = S^1$ commutes with the map $e^{i\theta}\mapsto e^{ip\theta}$ for $e^{i\theta}\in S^1$. In particular, $\sigma_{n - 1}$ commutes with $\sigma_2$ under the diffeomorphism $\varphi$ used to glue the pieces together, thus they paste together to give an $\mathcal{T}$-structure on $M_{\varphi}$.
\end{proof}

Let us recall the precise definition of the main cut-and-paste topological construction used to equip manifolds with  a $\mathcal{F}$-structure.

\begin{definition}{\label{Definition CS}} \cite[Section 3]{[HS]}. Let $M_i$ be closed smooth oriented n-manifolds, let $i_i: S^1\times D^{n - 1} \hookrightarrow M_i$ ($i = 1, 2$) be orientation-preserving embeddings that represent an element in $\pi_1(M_i)$, and let $\phi: S^{n - 2}\rightarrow S^{n - 2} = \partial D^{n - 1}$ be an orientation-reversing diffeomorphism. The manifold
\begin{equation}
M_1 \#_{S^1} M_2: = (M_1 - i_1(S^1\times int(D^{n - 1}))) \cup_{\phi}  (M_2 - i_2(S^1\times int(D^{n - 1})))
\end{equation}

is the \emph{circle sum of $M_1$ and $M_2$}, where the gluing map identifies the boundaries by $i_1(\theta, p) \approx i_2(\theta, \phi(p))$ for $\theta \in S^1$ and $p \in \partial D^{n - 1} = S^{n - 2}$.
\end{definition}

It is proven in \cite[Theorem 5.3]{[PP1]} that $\mathcal{T}$-structures behave well under connected sums. We extend that result as follows.

\begin{proposition}{\label{Proposition S1}}Let $M_i$ be closed oriented smooth n-manifolds $i\in \{1, 2\}$ with $n\geq 4$, which admit a non-trivial circle action. Their circle sum \begin{equation}M:= M_1 \#_{S^1} M_2\end{equation} admits a $\mathcal{T}$-structure.
\end{proposition}

The proof of Proposition \ref{Proposition S1} builds on arguments used in the proofs of \cite[Theorems 5.9 and 5.14]{[PP1]}.

\begin{proof} Fix an $i\in \{1, 2\}$,  and denote by $\gamma_i \subset M_i$ a loop that represents an element of the group $\pi_1(M_i)$ with tubular neighborhood $\nu(\gamma_i) = S^1\times D^{n - 1}$. The circle sum of Definition \ref{Definition CS} is expressed $M:= M_1 - \nu(\gamma_1) \cup M_2 - \nu(\gamma_2)$. Since we assume that the $M_i$ has dimension greater than three, $\gamma_i\subset M_i$ is of codimension at least three, and thus any two homotopic embeddings are isotopic \cite[Example 4.1.3]{[GS]}. For the $\mathcal{T}$-structure defined by the circle action, the loop $\gamma_i\subset M_i$ is a \emph{completely transversal $1$-sphere} \cite[Definition 5.11]{[PP1]}. An isotopy class of $\gamma_i$ intersects exactly one open subset of the $\mathcal{T}$-structure. We take the union of orbits of the circle action that go through $\gamma_i$, and pick its tubular neighborhood \begin{equation}S^1\times S^1\times D^{n - 2}.\end{equation} Deconstruct the  $(n - 2)$-disk $\{pt\}\times \{pt\}\times D^{n - 2}\subset S^1\times S^1\times D^{n - 2}$ into an inner ball and an outer annulus
\begin{equation}D^{n - 2} = D^{n - 2}_{\epsilon_1} \cup (S^{n - 3} \times [\epsilon_1, \epsilon_2]).\end{equation}

Use this decomposition to express the circle sum $M_1\#_{S^1} M_2$ as \begin{center}
$M_1 - (S^1\times S^1 \times D^{n - 2}) \cup (S^1\times S^1 \times D^{n - 2}_{\epsilon_1}) \cup$\\ $\cup (S^1\times S^1 \times S^{n - 3} \times [\epsilon_1, \epsilon_2] - S^1\times S^{n-3}\times D^2) \cup (M_2 - S^1\times S^1 \times D^{n - 2})$.\end{center}

Here, the $2$-disk $\{pt\}\times \{pt\} \times D^2 \subset S^1\times S^{n-3} \times D^2$ is a small 2-dimensional ball that is centered at a point in the middle of $S^1\times S^{n - 3}\times [\epsilon_1, \epsilon_2]$, and transverse to $S^1\times S^{n-3} \subset S^1\times S^1\times S^{n-3}\times [\epsilon_1, \epsilon_2]$. Equip \begin{equation}M_1 - S^1\times S^1\times D^{n - 2}_{\epsilon_1} \cup S^1\times S^1 \times D^{n - 2}_{\epsilon_1}\end{equation} and \begin{equation}M_2 - S^1\times S^1 \times D^{n - 2},\end{equation} with the initial $\mathcal{T}$-structure. For the $\mathcal{F}$-structure on the piece \begin{equation} S^1\times S^1 \times S^{n -3} \times [\epsilon_1, \epsilon_2] - S^1\times S^{n - 3} \times D^2,\end{equation} consider a nontrivial circle action on the $S^{n - 3}$ factor. The action induced on the $S^1\times S^{n-3}\times S^1$-boundary of each piece glues to a nontrivial circle action on the $S^1$-factors, and yields a $T^3$-action. This gives us an $\mathcal{T}$-structure on the circle sum $M = M_1\#_{S^1} M_2$ as it was claimed. 
\end{proof}

\begin{corollary}{\label{Corollary Fstr}} Let $M$ be a closed smooth oriented 5-manifold of fibered type and with fundamental group of order two and torsion-free second homology. Then, $M$ admits a $\mathcal{T}$-structure. 

\end{corollary} 

\begin{proof}  According to the classification of Hambleton-Su \cite[Theorem 3.7]{[HS]}, such a 5-manifold $M$ can be constructed as a circle sum along the loop that represents the nontrivial element of the fundamental group (see Definition \ref{Definition CS}) of copies of the following building blocks
\begin{itemize}
\item $S^2\times \mathbb{RP}^3$, 
\item $k(S^2\times S^2)\times S^1, \mathbb{CP}^2\times S^1$, and
\item $X^5(q)$ for $q \in \{1, 3, 5, 7\}$.
 \end{itemize}
 
The manifolds $X^5(q)$ are homotopy equivalent to $\mathbb{RP}^5 = X^5(1)$; in particular, we follow the notation of \cite{[HS]}. The corollary follows from Proposition \ref{Proposition S1} and by observing that there are $\mathcal{T}$-structures on the building blocks for which the loops involved in the cut-and-paste construction are \emph{completely transversal} 1-spheres in the sense of \cite[Definition 5.11]{[PP1]}. In particular, $S^1\times S^1\times D^3$ is a tubular neighborhood of the union of orbits that go through the loop. This enables the use of the $\mathcal{T}$-structure chosen in the proof of Proposition \ref{Proposition S1} for the building blocks of the cut-and-paste construction so that they glue appropriately to a final $\mathcal{T}$-structure on $M$.

Take a nontrivial circle action on $S^2$ and on $\mathbb{CP}^2$. It induces $\mathcal{T}$-structures on $S^2\times \mathbb{RP}^3$, on $k(S^2\times S^2)\times S^1$ (\cite[Theorem 5.9]{[PP1]}), and on $\mathbb{CP}^2\times S^1$. For these choices of $\mathcal{T}$-structures the loop that represents the nontrivial element in $\pi_1(S^2\times \mathbb{RP}^3) = \pi_1(\mathbb{RP}^3)$, $\{pt\}\times S^1 \subset k(S^2\times S^2)\times S^1$, and $\{pt\}\times S^1 \subset \mathbb{CP}^2\times S^1$ are completely transversal 1-spheres \cite[Example 5.13]{[PP1]}.

Regarding the existence of $\mathcal{T}$-structures satisfying the hypothesis of Proposition \ref{Proposition S1} for $X^5(q)$, Grove-Ziller \cite{[GZ]} have shown the existence of circle actions on all homotopy real projective 5-spaces as cohomogeneity one manifolds. An explicit $\mathcal{T}$-structure (and in particular, a circle action) on $X^5(q)$ can be constructed using the description of these manifolds as quotients of a fixed point free involution on a Brieskorn manifold given by Geiges-Thomas \cite[Section 4]{[GT]}. Consider the unit sphere $S^7\subset \C^4$ and the hypersurface $V_q^6 := \{z\in \C^4: z_0^q + z_1^2 + z_2^2 + z_3^2 = 0\}\subset \C^4$. The Brieskorn manifold $\Sigma^5_q$ is obtained as the intersection  $V_q^6 \cap S^7$. Take the fixed-point free orientation preserving involution
$T: \Sigma^5_q\rightarrow \Sigma^5_q$ given by $(z_0, z_1, z_2, z_3)\mapsto (z_0, - z_1, - z_2, - z_3)$. The diffeomorphism type of the orbit space is determined by $q \in \{0, 1, \ldots, 7, 8\}$, and $\Sigma^5_q/T = X^5(q)$ for odd values \cite{[GT]}.  A $\mathcal{T}$-structure is defined in the quotient by defining one on $\Sigma^5_q$ that commutes with the involution $T$, such as a circle action by complex multiplication. The claim follows from Proposition \ref{Proposition S1}. 
\end{proof}

\subsection{Topological entropy of products, submersions, immersions, and bi-invariant and symmetric metrics}{\label{Section Entropy}} We begin the section by extending a result of Paternain-Petean \cite[Lemma 2.4]{[PP1]} to compute the topological entropy of Riemannian metrics on orbit spaces and their universal covers as follows.

\begin{lemma} {\label{Lemma PP}} Paternain-Petean \cite{[GPa], [PP1]}
\begin{enumerate}
\item Let $(M_1, g_1)$ and $(M_2, g_2)$ be two compact Riemannian manifolds. Endow $M_1\times M_2$ with the product metric $g_1+ g_2$. Then
\begin{center}
$h_{top}(g_1 + g_2) = \sqrt{[h_{top}(g_1)]^2 + [h_{top}(g_2)]^2}$.
\end{center}
\item Let $(M, g_M) \rightarrow (N, g_N)$ be a Riemannian submersion where $M$ and $N$ are compact manifolds. Then \begin{center}$h_{top}(g_M) \geq h_{top}(g_N)$.\end{center}

\item If $\pi_1(N)$ is finite, and $M\rightarrow N$ is the universal cover then there exist Riemannian metrics $g_N$ and $g_M$ on $N$ and $M$ respectively such that \begin{center}$h_{top}(g_M) = h_{top}(g_N)$.\end{center}
\end{enumerate}
\end{lemma}

\begin{proof} Items (1) and (2) were proven in \cite[Lemma 2.4]{[PP1]}. We prove Item (3). Pick a Riemannian metric $g_N$ on $N$. There exists a unique Riemannian metric $g_M$ on $M$ for which $(M, g_M) \rightarrow (N, g_N)$ is a Riemannian covering. Item (2) implies $h_{top}(g_M) \geq h_{top}(g_N)$. We claim that the inequality in the other direction holds too by appealing to the following general situation. Let $\phi^{i}_t: X_i\longrightarrow X_i$ for $i= 1, 2$ be two flows and let $\pi:X_1\longrightarrow X_2$ be a continuous map that commutes with $\phi^i_t$, i.e., $\phi^2_t\circ \pi = \pi \circ \phi^1_t$. If $\pi$ is finite to one, then $h_{top}(\phi^1) \leq h_{top}(\phi^2)$ \cite[Proposition 3.15]{[GPa]}. Since the Riemannian covering $\pi$ winds the geodesic flows and it is finite-to-one by our hypothesis on the group $\pi_1(N)$, we conclude $h_{top}(g_M) \leq h_{top}(g_N)$. Therefore, the claimed equality holds. 
\end{proof}

The following result will be used in the proof of Theorems \ref{Theorem 4D} and \ref{Theorem 3}.

\begin{proposition}{\label{Proposition CROSS}} The symmetric metric of a compact symmetric space of rank one and any bi-invariant metric of a compact Lie group have zero topological entropy.

\end{proposition}

A proof of Proposition \ref{Proposition CROSS} different from the one below is given in \cite[p. 56]{[GPa]}, where Paternain argues that for these symmetric spaces, the counting function of the number of geodesic arcs joining two points has linear growth.

\begin{proof} We first claim that the Lyapunov exponents of the symmetric metric and the bi-invariant metric vanish. The geodesics of a compact symmetric space of rank one equipped with the symmetric metric are closed \cite[Remark 3.31]{[Be]} \cite[Proposition 5.3, Ch. IX.5]{[He]}, and they have the same period \cite[Chapter 3 and Section 7.C]{[Be]}.  The Jacobi equation \cite[Section 1.5]{[GPa]} implies that the Lyapunov exponents $\lambda_j$  of any geodesic vanish. Analogously, the Jacobi fields of a bi-invariant metric on a compact Lie group $G$ have linearly bounded growth since every geodesic is given as the orbit of a 1-parameter subgroup of the compact group $G$. Given that the Jacobi fields grow at most linearly, the Jacobi equation \cite[Section 1.5]{[GPa]} implies that every Lyapunov exponent $\lambda_j$ vanishes. 

Let $g$ be either the symmetric metric or a bi-invariant metric. We now relate the Lyapunov exponents with the entropy $h_{\mu}(\phi_{g})$, where $\mu$ is an ergodic measure, and $\phi_g$ the geodesic flow. Following \cite[Section 3.2.4]{[GPa]}, set $\lambda^+ := \sum \lambda_j$. Employ Ruelle's inequality \cite[Theorem 2]{[R]} and \cite[Section 3.2.4]{[GPa]} to obtain
\begin{equation}
h_{\mu}(\phi_g) \leq \int_M \lambda^+ d \mu,
\end{equation}

and thus $h_{\mu}(\phi_g) = 0$, i.e., all measure entropies are zero. By the variational principle \cite[p. 62, Equation (3.4)]{[GPa]}\begin{equation}h_{top}(g) = \sup h_{\mu}(\phi_{g}) = 0,\end{equation}

where the supremum of $h_{\mu}$ is taken over all ergodic measures. Therefore, we have $h_{top}(g) = 0$.
\end{proof}

From a metric with zero topological entropy one can produce a new non-isometric metric, which also has zero topological entropy by using Cheeger deformations \cite{[C]} as follows.

\begin{remark} Let $(M, g)$ be a Riemannian manifold with $h_{top}(g) = 0$, and let $G$ be a Lie group that acts by isometries on $(M, g)$ as $\tilde{g}\star (p, g) = (\tilde{g} p, \tilde{g} g)$. The quotient is $M = (M\times G)/ \Delta G$. The new metric $g'$ on $M$ is a quotient metric of the product of the original metric $g$ with a bi-invariant metric of $G$ \cite[Example 2]{[C]} (cf. \cite[Section 2]{[Z]}). Lemma \ref{Lemma PP} and Proposition \ref{Proposition CROSS} imply $h_{top}(g') = 0$.
\end{remark}

We now observe how the existence of a metric with zero topological entropy restricts the choices of homotopy types of the fibered 5-manifolds considered in the following proposition.

\begin{proposition}{\label{Proposition H5D}} A closed smooth orientable fibered 5-manifold $M$ with fundamental group $\pi_1(M) = \Z/2$ and torsion free $H_2(M; \Z)$ admits a Riemannian metric $g$ with $h_{top}(g) = 0$ if and only if $M$ is homotopy equivalent to $\mathbb{RP}^5$, an $S^2$-bundle over $\mathbb{RP}^3$ or the orbit space of $S^3\widetilde{\times} S^2$ by a $\Z/2$-involution.
\end{proposition}

\begin{proof} Assume that there exists a Riemannian meric $(M, g)$ such that $h_{top}(g) = 0$.  Denote by $\tilde{g}$ be the unique Riemannian metric such that $(\widetilde{M}, \tilde{g})\rightarrow (M, g)$ is a Riemannian submersion. Lemma \ref{Lemma PP} allows us to conclude that $h_{top}(\tilde{g}) = 0$, and a result of Paternain-Petean \cite[Theorem E]{[PP1]} says that the possible diffeomorphism types of $\widetilde{M}$ are $S^5, S^3\times S^2$, and the nontrivial $S^3$-bundle over $S^2$. Depending on the value of the second Stiefel-Whitney classes $\omega_2(\widetilde{M})$ and $\omega_2(M)$, the classification result of Hambleton-Su \cite[Theorem 3.12]{[HS]} implies that $M$ is homotopy equivalent to the real projective 5-space, an 2-sphere bundle over the real projective 3-space or to a quotient of $S^3\widetilde{\times} S^2$ by a $\Z/2$-involution. Representatives of these homotopy types are \begin{equation} \mathbb{RP}^5, \mathbb{RP}^5\#_{S^1} \mathbb{RP}^5, S^2\times \mathbb{RP}^3, \mathbb{RP}^5\#_{S^1}(\mathbb{CP}^2\times S^1). \end{equation}

We proceed to argue that these four manifolds admit a Riemannian metric of zero topological entropy using Lemma \ref{Lemma PP} and Proposition \ref{Proposition CROSS}. The claim for $\mathbb{RP}^5$ and $S^2\times \mathbb{RP}^3$ follows by using their symmetric metrics. For the remaining two cases, a submersion metric has vanishing topological entropy. Indeed, $X^5(0) = \mathbb{RP}^5\#_{S^1}\mathbb{RP}^5$ is diffeomorphic to the orbit space of the product $S^2\times S^3$ by the orientation-preserving free involution that acts as the the antipodal map on the $S^3$-factor, and as the reflection along a line on the $S^2$-factor \cite[Remark 3.9]{[HS]}. The manifold $\mathbb{RP}^5\#_{S^1} (\mathbb{CP}^2\times S^1)$ is the quotient of $S^3\widetilde{\times} S^2$ by a linear action of $\Z/2$ cf. \cite[Lemma 4.1]{[PP1]}.

\end{proof}

\section{Proofs}

\subsection{Proof of Theorem \ref{Theorem Extra}}{\label{Section K}} Consider the inequivalent smooth structures that are constructed using Knot surgery \cite{[FS]}; the claims regarding the infinite set $\{E(n)_s: s\in \N\}$ follow from verbatim arguments. Proposition \ref{Proposition Knot} says that every manifold that is contained in the infinite set $\{E(n)_K:$ K is a knot in $S^3\}$ admits a $\mathcal{T}$-structure. We have then the equalities $h(E(n)_K) = 0 = Vol_K(E(n)_K)$ by \cite[Theorems A and B]{[PP1]}, and all these inequivalent smooth structures collapse with curvature bounded from below. We now prove the claim $\mathcal{Y}(E(n)_K) = 0$. First, the Yamabe invariant of these manifolds is nonnegative by \cite[Theorem 7.2]{[PP1]}. Indeed, the collapse obtained from the existence of the $\mathcal{T}$-structure implies \begin{equation}0 = Vol_{K}(M) = Vol_{Ric}(M) = Vol_{|Scal|}(M) = Vol_{Scal},\end{equation} and the last equality is equivalent to $\mathcal{Y}(M) \geq 0$. Since the Seiberg-Witten invariant \begin{equation} SW_{E(n)_K} = SW_{E(n)}\cdot \Delta_K \end{equation} where $\Delta_K$ is the Alexander polynomial \cite{[FS]}, these invariants are nontrivial, and $E(n)_K$ does not admit a Riemannian metric of positive scalar curvature \cite[Section 3]{[W]}, i.e., their Yamabe invariant is nonpositive. Thus, $\mathcal{Y}(E(n)_K) = 0$, as claimed.

If such a Riemannian manifold $(M, g)$ has $scal_g = 0$, then it is Ricci flat. Since $M$ has the homotopy type of an elliptic surface, \begin{equation} c_1^2(M) = 2c_2(M) + 3\sigma(M) = 0\end{equation} where $c_2$ is the Euler characteristic and $\sigma$ the signature. Using the generalized Gauss-Bonnet and the signature theorems \cite[Chapter 6, Section D]{[Be1]} \begin{equation} 0 = \int_M (\frac{scal_g^2}{24} - \frac{|\overset{\circ}{Ric}|^2}{2} + 2|W^+|)d\mu_g, \end{equation} we conclude that $|W^+| = 0$ and $M$ is half-conformally flat, since for a Ricci-flat manifold both the scalar curvature and the traceless Ricci tensor vanish. A scalar-flat half-conformally flat Riemannian 4-manifold is diffeomorphic to the K3 surface $E(2)$ \cite[13.30 Theorem]{[Be1]} (cf. \cite{[H]}, \cite[6.35 and 6.40 Theorems]{[Be1]}).

\subsection{Proof of Theorem \ref{Theorem 4D}}{\label{Section Proof4D}} Let $L_{\Z/p}$ be the closed smooth 4-manifold with fundamental group $\Z/p$ for $p\geq 2$ and the rational homology of the 4-sphere constructed in Proposition \ref{Proposition C1}. In particular, its Euler characteristic is two, and its signature zero. A closed smooth orientable 4-manifold with finite cyclic fundamental group is homeomorphic either to \begin{equation} a\mathbb{CP}^2 \# b \overline{\mathbb{CP}^2}\#L_{\Z/p},\end{equation} \begin{equation} aE(2)\#b(S^2\times S^2)\#L_{\Z/p},\end{equation} or to the orbit space of an involution \begin{equation} (aE(n)\# b(S^2\times S^2))/\Z/p\end{equation} 

according to Hambleton-Kreck \cite[Theorem C]{[HK2]}. In order to display a smooth structure on each homeomorphism class equipped with an $\mathcal{F}$-structure, we can proceed in different ways. 
The claim for the first two homeomorphism types follows from taking connected sum of copies of $\mathbb{CP}^2, \overline{\mathbb{CP}^2}, K3, S^2\times S^2$, and $L_{\Z/p}$, which have an $\mathcal{F}$-structure \cite[Theorem 5.9]{[PP1]}. Moreover, we can alternatively use circle sums as in Definition \ref{Definition CS}, and Proposition \ref{Proposition S1}.
To exhibit a smooth structure on the homeomorphism type of orbit spaces of the involution with fundamental group $\Z/2$ that has an $\mathcal{F}$-structure, we use the Enriques surface. This complex surface is the orbit space $S_4/\varphi$ of the hypersurface of degree 4 \begin{equation}S_4:= \{[z_0:z_1:z_2:z_3]\in \mathbb{CP}^3 : z_0^4 + z_1^4 + z_2^4 + z_3^4 = 0\} \subset \mathbb{CP}^3\end{equation} by complex conjugation $\varphi:S_4\rightarrow S_4$. It is an elliptic surface \cite[Section 3.4]{[GS]}, and \cite[Theorem 5.10]{[PP1]} implies that it admits an $\mathcal{F}$-structure. Therefore $h(M) = 0$ by \cite[Theorem A]{[PP1]}. 

The claim about the Yamabe invariant follows from \cite[Theorem 7.2]{[PP1]}, as it was argued in Section \ref{Section K}. We do point out that \begin{equation} \mathcal{Y}(a\mathbb{CP}^2 \# b \overline{\mathbb{CP}^2}\# c(S^2\times S^2)\#L_{\Z/p}) > 0\end{equation} since these connected sums admit a Riemannian metric of positive scalar curvature \cite{[GL], [SY]}. On the other hand, the homeomorphism types of nonvanishing signature do not admit such a metric by Lichnerowicz $\hat{A}$-genus obstruction \cite[Theor\`eme 2]{[Li]}.

Suppose there exists a Riemannian metric $g$ on $M$ such that $h_{top}(g) = h(M)$. Since the minimal entropy is zero, the topological entropy of the metric vanishes. Consider the universal cover $\widetilde{M}$ and the unique Riemannian metric $\tilde{g}$ that makes $(\widetilde{M}, \tilde{g})\rightarrow (M, g)$ a Riemannian submersion. Lemma \ref{Lemma PP} implies $h_{top}(\tilde{g}) = 0$. The manifold $\widetilde{M}$ is homeomorphic to $S^2\times S^2$ by \cite[Theorem D]{[PP1]}, and $M$ is homeomorphic to a rational homology 4-sphere with fundamental group of order two \cite[Theorem C]{[HK2]}. The total spaces of the two orientable $S^2$-bundles over $\mathbb{RP}^2$ are diffeomorphic to quotients \begin{equation}S^2\times S^2/(A,  - I),\end{equation} where $A\in Isom(S^2) = O(3)$ with det A $= - 1$, and $-I: S^2\times S^2$ is the antipodal involution. These manifolds are rational homology 4-spheres with fundamental group of order two, and they can be distinguished by their second Stiefel-Whitney classes \cite[Chapters 12.2 and 12.3]{[H]}. Let $g_{S^n}$ be the round metric on the n-sphere: then $h_{top}(g_{S^n}) = 0$ by Proposition \ref{Proposition CROSS}. It follows from Lemma \ref{Lemma PP} that both quotient metrics have zero topological entropy. This concludes the proof of the theorem.

\subsection{Proof of Theorem \ref{Theorem 3}}  As it was mentioned in the proof of Corollary \ref{Corollary Fstr}, the classification results of Hambleton and Su \cite[Theorem 3.7]{[HS]} say that every 5-manifold that satisfies the hypothesis of Theorem \ref{Theorem 3} can be constructed as a circle sum (see Definition \ref{Definition CS}) of copies of the following building blocks
\begin{itemize}
\item $S^2\times \mathbb{RP}^3$, $k(S^2\times S^2)\times S^1$, $\mathbb{CP}^2\times S^1$, and
\item $X^5(q)$ for $q \in \{1, 3, 5, 7\}$.
 \end{itemize}

Let us begin by proving the claim about the scalar curvature and Yamabe invariant. All the building blocks have Riemannian metrics of positive scalar curvature. Indeed, let $g$ be either the round metric on the n-sphere $g_{S^n}$, the quotient metric on the real projective n-space $g_{\mathbb{RP}^n}$, or the Fubini-Study metric on the complex projective space $g_{\mathbb{CP}^2}$. The product metrics $g + dt^2$ and $g + g$ have positive scalar curvature; here, $dt^2$ is the canonical metric on the circle. All four homotopy types of real projective 5-space admit a Riemannian metric of positive scalar curvature \cite{[KS]}. The claim regarding positive scalar curvature now follows from \cite{[GL], [SY]}, since the loops along which the cut-and-paste construction is done are submanifolds of codimension four. Moreover, the Yamabe invariant $\mathcal{Y}(M)$ is positive if and only if $M$ has a metric of positive scalar curvature. If the universal cover is not spin, then the claim was proven in \cite[Theorem 2.4]{[BR2]}.

The existence of a $\mathcal{F}$-structures was proven in Proposition \ref{Proposition S1} and Corollary \ref{Corollary Fstr}, thus impliying $h(M) = 0$ by \cite[Theorem A]{[PP1]}.

Suppose there exists a Riemannian meric $(M, g)$ such that $h_{top}(g) = h(M)$. Since $h(M) = 0$, the metric has zero topological entropy. Let $\tilde{g}$ be the unique Riemannian metric such that $(\widetilde{M}, \tilde{g})\rightarrow (M, g)$ is a Riemannian submersion. We have $h_{top}(\tilde{g}) = 0$ by Lemma \ref{Lemma PP}, and \cite[Theorem E]{[PP1]} states that the possible diffeomorphism types of $\widetilde{M}$ are $S^5, S^2\times S^3$, and the nontrivial $S^3$-bundle over $S^2$. Notice that since $H_2(SU(3)/SO(3)) \cong \Z/2\Z$, our hypothesis on the lack of torsion in $H_2(M)$ excludes the Wu manifold as a universal cover for the manifolds under consideration. Thus, $b_2(M)\leq 1$. The claim on the restriction of the homotopy types and the existence of a metric with zero topological entropy on $S^2\times \mathbb{RP}^3$ was proven in Proposition \ref{Proposition H5D} . If $\omega_2(M) = 0$, \cite[Theorem 3.7]{[HS]} states that the manifold $M$ is diffeomorphic to $S^2\times \mathbb{RP}^3$. This finishes the proof of the theorem. 

\subsection{Proof of Theorem \ref{Theorem 6D}} Smale showed that a 2-connected closed smooth 6-manifold is homeomorphic to $S^6$ or to a connected sum $k(S^3\times S^3)$ \cite[Theorem B]{[Sm]}. According to the classification of Hambleton \cite{[IH]} of orientation preserving $\Z/2$-involutions on connected sums $k(S^3\times S^3)$ with $k\in \N$, the homotopy types of orbit spaces are realized as connected sums of $S^3\times S^3$ and either $S^3\times \mathbb{RP}^3$ or the twisted $S^3$-bundle over $\mathbb{RP}^3$ (cf. \cite[Section 1]{[Wa]}). We denote the latter by $\eta_{S^3}$.

The Hopf action on $S^3$ commutes with the structure group of $\eta_{S^3}$, thus both $S^3$-bundles over the real projective 3-space and $S^3\times S^3$ admit a free circle action. The existence of a polarized $\mathcal{F}$-structure given by this action implies $MinVol(M) = 0$ by a result of Cheeger-Gromov \cite{[CG]}, and $M$ collapses with bounded sectional curvature. The existence of a $\mathcal{T}$-structure on every homotopy type now follows from \cite[Theorem 5.9]{[PP1]}. The claim regarding the Yamabe invariant follows from the existence of a metric of positive scalar curvature using the results \cite{[GL], [SY]}.

Regarding the existence of metrics of vanishing topological entropy, we have the following. The special orthogonal group $SO(4)$ is diffeomorphic to $S^3\times \mathbb{RP}^3$, and a bi-invariant metric $g$ on it has $h_{top}(g) = 0$ by Proposition \ref{Proposition CROSS}. Lemma \ref{Lemma PP} and Proposition \ref{Proposition CROSS} imply that the round metric on $S^6$ and the product of round metrics on $S^3\times S^3$ have zero topological entropy. The isometry group of $S^3\times S^3$ is $O(4)\times O(4) \rtimes \Z/2$. Mimicking the proof of Theorem \ref{Theorem 4D} and Proposition \ref{Proposition H5D}, the total space of the sphere bundle $\eta_{S^3}$ can be expressed as $S^3\times S^3/(A, -I)$, where $A\in Isom(S^3) = O(4)$, and $-I: S^3\rightarrow S^3$ is the antipodal involution. Thus, $\eta_{S^3}$ admits a metric of zero topological entropy.

We now prove that the existence of a metric of zero topological entropy restricts the possible homotopy types to the ones of \emph{rationally elliptic spaces} \cite[Chapter VI]{[FHT]}, whose definition we now recall. Let $M$ closed smooth simply connected manifold, and let $\Omega M$ denote the space of based loops on $M$. The manifold $M$ is \emph{rationally elliptic} if the homology of the loop space\begin{equation}\overset {n} {\underset {i = 0} {\sum}} dim H_i(\Omega M; \Q),\end{equation}
grows polynomially with $n$. We claim that a 2-connected 6-manifold $M$ admits a Riemannian metric of zero topological entropy if and only if it is rationally elliptic, i. e., it is diffeomorphic to either $S^6$ or $S^3\times S^3$. Suppose $M$ is a rationally elliptic space. The Hurewicz isomorphism implies $\pi_3(M) \cong H_3(M; \Q)$, and in particular \begin{equation} b_3 = dim H_3(M; \Q) = dim(\pi_3(M)\otimes \Q).\end{equation} It is shown in \cite[Corollary 1.3]{[FH]} that if $X^n$ is such a rationally elliptic manifold $X^n$ then \begin{equation}
\underset {k}{\sum} (2k - 1) dim (\pi_{2k - 1}(X^n) \otimes \Q) \leq 2n - 1\end{equation} holds. Thus, the third Betti number of an elliptic closed 2-connected 6-manifold obeys the bound $b_3(M)< 4$.  The only two such diffeomorphisms types are $S^6$ and $S^3\times S^3$. Both these manifolds are rationally elliptic spaces, and we have proven that they both admit Riemannian metrics of zero topological entropy.  To prove the converse, we proceed by contradiction. Suppose that we have Riemannian metric on $M$ with vanishing topological entropy, and assume $M$ is not rationally elliptic. It is known that the homology of the loop space with rational coefficients grows either polynomially or exponentially \cite{[FHT]}. A result of Gromov states that if the growth were exponential, then any metric $g$ on $M$ satisfies $h_{top}(g) > 0$ \cite[Theorem 2.3]{[GPa1]}, a contradiction to our initial hypothesis. 

Finally, suppose that $M$ has fundamental group of order two and that it admits a metric of vanishing topological entropy. Lemma \ref{Lemma PP} implies that so does its universal cover $\widetilde{M}$, which is diffeomorphic to $S^3\times S^3$. The results of Hambleton \cite{[IH]} imply that $M$ is homotopy equivalent to $S^3\times \mathbb{RP}^3$ or $\eta_{S^3}$.

\subsection{Proof of Theorem \ref{Theorem 1}}{\label{Section PT1}} We summarize the recent progress on unveiling infinitely many inequivalent smooth structures on simply connected 4manifolds of negative signature (see \cite{[FPS], [ABBKP], [BK], [AP1]} for the nonspin case, and \cite{[PS]} for the spin one) as follows. Let $(e, \sigma) \in \Z\times \Z$ satisfy \begin{center}$2e + 3\sigma \geq 0$ and $\frac{1}{4}(e + \sigma) \in \Z$,\end{center} and suppose that $\sigma \geq -1$. There exists an infinite set $\{X_n: n\in \N\}$ of closed simply connected homeomorphic 4-manifolds such that \begin{center}$e(X_n) = e$, $\sigma(X_n) = \sigma$, and $SW_{X_i}\neq SW_{X_j}$\end{center} for $i\neq j$.

We make use of the existence of these manifolds, and consider the set \begin{equation}\{X_n\# L_G : n\in \N\}\end{equation} of homeomorphic manifolds. The manifold $L_G$ is a rational homology 4-sphere with fundamental group $G$ of order $d$ that was built in \cite{[Ue]}, \cite[Section 3]{[HKW]}, and we recall its construction. Let $G$ be a finite group of order $d \geq 2$ that acts freely on $S^3$. Take the product $S^3/G\times S^1$, let $\gamma:= \{pt\}\times S^1$ be the loop that generates the infinite cyclic factor of the group $\pi_1(S^3/G\times S^1) = G\times \Z$, and denote its tubular neighborhood by $\nu(\gamma) \cong D^3\times S^1$. The rational homology sphere is \begin{equation}
L_G:= (S^3/G \times S^1 - \nu(\gamma)) \cup (D^2\times S^2), 
\end{equation}
with $\omega_2(L_G) = 0$. Its universal cover $\widetilde{L_G}$ is diffeomorphic to $(d - 1)(S^2\times S^2)$.

A result of Kotschick-Morgan-Taubes \cite[Proposition 2]{[KMT]} implies that \begin{equation}SW_{X_i\# L_G} \neq SW_{X_j\# L_G}\end{equation} if $i\neq j$. In particular, $\{X_n\#L_G : n\in \N\}$ is an infinite set of pairwise nondiffeomorphic yet homeomorphic 4-manifolds with Euler characteristic $e(X_n)$, signature $\sigma(X_n)$, fundamental group $G$, and second Stiefel-Whitney class $\omega_2(X_n)$.

We claim that the universal covers $\widetilde{X_n\#L_G}\rightarrow X_n\#L_G$ have the standard smooth structure $\forall n\in \N$. Looking at the homeomorphism types of the simply connected manifolds $\{X_n: n\in \N\}$ we have the following \cite{[F]}. If $\omega_2(X_n) \neq 0$, the homeomorphism type is \begin{equation} (2m_1 + 1)\mathbb{CP}^2 \# (2m_1 + 2)\overline{\mathbb{CP}^2}\end{equation} for $m_1\in \N$. If $\omega_2(X_n) = 0$, the homeomorphism type is \begin{equation}(E(2q))\#2k(S^2\times S^2)\end{equation} for $q\geq 1$, and $k\geq 0$.

It was proven in \cite{[BS]} that there exists a diffeomorphism \begin{equation}X_i\#S^2\times S^2 \rightarrow X_j\#S^2\times S^2\end{equation} $\forall i, j \in \N$ if $\omega_2(X_i)\neq 0$. It was proven in \cite{[T1]} that the existence results of infinitely many inequivalent smooth structures on simply connected spin 4-manifolds of negative signature \cite{[PS]} can be proven using logarithmic transformations as in \cite{[FPS], [BK], [ABBKP]}. Thus, the results in \cite{[BS]} apply for both nonspin and spin manifolds. Since $\widetilde{L_G}$ is diffeomorphic to the connected sum $(d - 1)(S^2\times S^2)$, and $E(2q)\# S^2\times S^2$ is diffeomorphic to $qE(2)\#q(S^2\times S^2)$ \cite{[Go]}, the theorem follows.

\subsection{Proof of Theorem \ref{Theorem 2}} The Fubini-Study metric on $\mathbb{CP}^2$ and $\overline{\mathbb{CP}^2}$, and the product metric  $g_{S^2} + g_{S^2}$ of the round metric on the 2-sphere on $S^2\times S^2$ have positive scalar curvature. Schoen-Yau \cite{[SY]} and Gromov-Lawson \cite{[GL]} have shown that the admission of a Riemannian metric of positive scalar curvature is closed under connected sums, and cut-and-paste constructions along submanifolds of codimension three. Thus, $L_G$ and the manifolds \begin{equation} (da + 1)\mathbb{CP}^2 \# (db + 1)\overline{\mathbb{CP}^2}\# (d - 1)(S^2\times S^2),\end{equation} admit such a metric.

As it was discussed in Section \ref{Section PT1}, the infinite set of pairwise nondiffeomorphic orbit spaces realizing the same homeomorphism type $\{X_n\#L_G: n\in \N\}$ have nontrivial Seiberg-Witten invariant \cite[Proposition 2]{[KMT]}, thus none of them admits a metric of positive scalar curvature \cite[Section 3]{[W]}. The standard smooth structure on $X_n\#L_G$ is a connected sum of copies of $\mathbb{CP}^2$ with either orientation and $L_G$, and thus admits a Riemannian metric of positive scalar curvature. 

Regarding the existence of infinitely many inequivalent smooth structures on the orbit space and on the universal covers that do not admit such a metric, we argue as follows. Infinite sets of pairwise non-diffeomorphic symplectic manifolds with fundamental group $\Z/p$ and $\Z/p\oplus \Z/q$ for every Euler characteristic and signature of the homeomorphic but not diffeomorphic manifolds described in the proof of Theorem \ref{Theorem 1}  were constructed in \cite{[BK], [AP1], [T2]}. Using the pullback of the symplectic form, one concludes that their universal covers also admit a symplectic structure. Taubes' theorem \cite{[Ta]} says that the Seiberg-Witten invariant of a symplectic manifold $(X, \omega)$ is $SW_X(\pm c_1(X, \omega)) = \pm 1$, and thus none of these manifolds admits a metric of positive scalar curvature \cite[Section 3]{[W]}.

\end{document}